\documentclass[11pt]{article}
\usepackage{amsmath}
\usepackage{amsfonts}
\usepackage{amsmath,amsthm}
\usepackage{amsmath,amssymb,amsthm,latexsym}
\usepackage{amscd}
\usepackage[all]{xy}
\usepackage{hyperref}

\newtheorem{theorem}{Theorem}[section]
\newtheorem{proposition}[theorem]{Proposition}
\newtheorem{definition}[theorem]{Definition}

\newtheorem{lemma}[theorem]{Lemma}

\newtheorem{example}[theorem]{Example}

\theoremstyle{remark}
\newtheorem{remark}[theorem]{Remark}

\begin{document}

\title{\bf{Complete stationary surfaces in $\mathbb{R}^4_1$ with total Gaussian curvature $-\int K\mathrm{d}M=6\pi$}}

\author{Xiang Ma
\thanks
{LMAM, School of Mathematical Sciences, Peking University,
100871 Beijing, People's Republic of China.
Fax: +86-10-62751801, Email: maxiang@math.pku.edu.cn \hspace{2mm} Funded by the Project 10901006 of
National Natural Science Foundation of China. The author thanks Sheng Bai for pointing out a mistake in the original proof to Lemma 4.1.}
}

\maketitle

\begin{center}
{\bf Abstract}
\end{center}

In a previous paper we classified complete stationary surfaces
(i.e. spacelike surfaces with zero mean curvature)
in 4-dimensional Lorentz space $\mathbb{R}^4_1$
which are algebraic and with total Gaussian curvature
$-\int K\mathrm{d}M=4\pi$. Here we go on with the study of
such surfaces with $-\int K\mathrm{d}M=6\pi$.
It is shown in this paper that the topological type of such a surface must be a M\"obius strip. On the other hand, new examples
with a single good singular end are shown to exist.


{\bf MSC(2000):\hspace{2mm} 53A10, 53C42, 53C45}

\section{Introduction}

In a previous paper \cite{Ma} we have generalized
the classical theory of minimal surfaces in $\mathbb{R}^3$ to
zero mean curvature spacelike surfaces in 4-dimensional Lorentz
space. Such an immersed surface ${\bf x}:M\to\mathbb{R}^4_1$
is called a \emph{stationary surface}.

Based on this general theory, in \cite{MaWang}
we classified those algebraic ones with least possible
total (Gaussian) curvature $-\int K\mathrm{d}M=4\pi$. This extends a classical result of Osserman that a complete minimal surface in
$\mathbb{R}^3$ with total curvature $4\pi$ is either the catenoid or the Enneper surface \cite{Osser}.

Another work in \cite{MaWang} is that we generalized the theory
of non-orientable minimal surfaces in $\mathbb{R}^3$
\cite{Meeks, Oliveira, Lopez, Martin} to our setting.
This is done by the same idea, namely, lifting everything
to the oriented double covering surface $\widetilde{M}$.

Compared with minimal surfaces in $\mathbb{R}^3$,
a major difference is that stationary surfaces in $\mathbb{R}^4_1$
with finite total curvature may fail to be algebraic \cite{Ma}.
By this term \emph{algebraic}, we mean that the Weierstrass data, or equivalently,
the vector-valued holomorphic differential ${\bf x}_z\mathrm{d}z$,
extend to meromorphic functions or forms defined on a compact Riemann surface.
Nevertheless, finite total curvature still implies
finite topology. So in this paper we assume that the double covering
$\widetilde{M}$ has genus $g$, and the open surface $M$ is
homomorphic to a closed non-orientable surface (with Euler
characteristic $1-g$) punctured at $r$ points (the ends).

By Gauss-Bonnet type formulas established in \cite{Ma},
we know that complete algebraic stationary surfaces with
$-\int K\mathrm{d}M=6\pi$ must be non-orientable.
On the other hand, we obtained a lower bound in \cite{MaWang} for any $g\ge 0$ as below:
\begin{equation}\label{eq-bound}
-\int_M K\mathrm{d}M\ge 2\pi(g+3).
\end{equation}
Thus it is natural to study complete,
algebraic, non-orientable stationary surfaces with total curvature $6\pi$, which is the least possible value.

By \eqref{eq-bound} and the generalized Jorge-Meeks formula (see Theorem ~\ref{GB2}), we know such surface $M$ must be
homomorphic to a projective plane with one or two ends.
Meeks have obtained

\bigskip
{\bf Known results in $\mathbb{R}^3$ \cite{Meeks}:} There is a unique
complete, immersed minimal surface with total curvature
$-\int K\mathrm{d}M=6\pi$, which is now called Meeks' M\"obius strip. In particular, the topological type of a projective plane with two punctures does not occur.
\bigskip

New examples appear in $\mathbb{R}^4_1$. We have constructed complete stationary M\"obius strips with total curvature $6\pi$ in \cite{MaWang}. Among them, one family is a generalization of Meeks and Olivaira's
examples; another family has an essential singularity at
the end. We will review the descriptions of them in Section~3.

Besides that, we are still interested to know other
possible examples in $\mathbb{R}^4_1$.
Can the topological type of a projective plane with
two punctures be realized as a stationary surface
in $\mathbb{R}^4_1$? On the other hand,
besides Meeks' example and its deformations as mentioned above,
does there exist new type of stationary M\"obius strips?
These two questions are answered at here.

\bigskip
{\bf Conclusion 1:} There exists a family of
complete, immersed stationary M\"obius strips in $\mathbb{R}^4_1$
with a good singular end and total curvature
$-\int_M K\mathrm{d}M= 6\pi.$
\bigskip

 Recall that an end is called
\emph{singular end} if the two Gauss maps satisfy $\phi=\bar\psi$.
Please see Definition~\ref{def-singular} in Section~2, or \cite{Ma} for detailed explanation.
This is a special phenomenon which never occur for
minimal surfaces in $\mathbb{R}^3$.
In particular, the total curvature integral converges around such an end if, and only if, these two functions take the same value with
different multiplicities. This case we call it a good singular end.

\bigskip
{\bf Conclusion 2:} There does not exist
examples with $-\int K\mathrm{d}M=6\pi$ and
homomorphic to the projective plane punctured at two points.
\bigskip

A noteworthy observation (Lemma~\ref{lem-flux}) is that the flux at any end of a
stationary M\"obius strip (with several punctures) must vanish.
It is interesting to know whether this is true for any other
non-orientable stationary surface.

This paper is organized as follows. In Section~2 we review the
Weierstrass representation and Gauss-Bonnet type formulas.
The examples with a good singular end are shown to exist in
Section~3, together with a discussion of other possible
examples. The non-existence of topological type of
the projective plane punctured at two points is contained in
Section~3, where we rule out it case by case.

\section{Preliminary}
Let ${\bf x}:M^2\to \mathbb{R}^4_1$ be an oriented complete
spacelike surface in 4-dimensional Lorentz space.
The Lorentz inner product $\langle\cdot,\cdot\rangle$ is given by
\[\langle {\bf x},{\bf x}\rangle=x_1^2+x_2^2+x_3^2-x_4^2.\]
We will briefly review the basic facts and global results
established in \cite{Ma} about such surfaces
with zero mean curvature (called \emph{stationary surfaces}).

Let $\mathrm{d}s^2=\mathrm{e}^{2\omega}|\mathrm{d}z|^2$
be the induced Riemannian metric on $M$ with respect to a
local complex coordinate $z=u+\mathrm{i}v$. Hence
\[
\langle {\bf x}_{z},{\bf x}_{z}\rangle=0,~~ \langle
{\bf x}_{z},{\bf x}_{\bar{z}}\rangle =\frac{1}{2}\mathrm{e}^{2\omega}.
\]
The Weierstrass-type
representation of stationary surface ${\bf x} : M\rightarrow \mathbb{R}^4_1$ is given by \cite{Ma}:
\begin{equation}\label{x}
{\bf x}= 2~\mathrm{Re}\int \Big(\phi+\psi, -\mathrm{i}(\phi-\psi),1-\phi\psi,1+\phi\psi\Big)\mathrm{d}h
\end{equation}
in terms of two meromorphic functions $\phi,\psi$ and
a holomorphic $1$-form $\mathrm{d}h$ locally.
 We call $\phi,\psi$ the Gauss maps of ${\bf x}$ and
$\mathrm{d}h$ the height differential as in $\mathbb{R}^3$.
Indeed, $\{\phi,\bar{\psi}\}$ correspond to
the two lightlike normal directions at each point of ${\bf x}(M)$.

\begin{remark}\label{rem-Wrepre}
When $\phi\equiv \mp 1/\psi$, by \eqref{x} we recover the Weierstrass representation formula for
a minimal surface in $\mathbb{R}^3$, or a maximal surface in
$\mathbb{R}^3_1$. When $\phi$ or $\psi$ is constant,
we get a zero mean curvature spacelike surface in the 3-space
$\mathbb{R}^3_0\triangleq \{(x_1,x_2,x_3,x_3)\in\mathbb{R}^4_1\}$
with an induced degenerate inner product, which is essentially
the graph of a harmonic function $x_3=f(x_1,x_2)$ on
complex plane $\mathbb{C}=\{x_1+\mathrm{i}x_2\}$.
From now on we always assume that neither of $\phi,\psi$
is a constant unless it is stated otherwise,
because in this degenerate case the induced metric is always flat and the total curvature is $0$.
\end{remark}

\begin{remark}\label{rem-trans}
A Lorentz orthogonal transformations of $\mathbb{R}^4_1$
induce a M\"obius transformation on $S^2$, or equivalently,
a fractional linear transformation on $\mathbb{C}P^1=\mathbb{C}\cup\{\infty\}$
given by $A=\left(\begin{smallmatrix}a & b \\ c & d\end{smallmatrix}\right)$
 with $a,b,c,d \in \mathbb{C},~ad-bc=1$.
The Gauss maps $\phi,\psi$ and the height differential $\mathrm{d}h$ transform as below:
\begin{equation}\label{trans}
\phi\Rightarrow
\frac{a\phi+b}{c\phi+d}~,~~
 \psi\Rightarrow \frac{\bar{a}\psi+\bar{b}}{\bar{c}\psi+\bar{d}}~,~~
 \mathrm{d}h\Rightarrow (c\phi+d)(\bar{c}\psi+\bar{d})\mathrm{d}h~.
 \end{equation}
This is used to normalize the representation of examples.
\end{remark}

\begin{theorem}\label{thm-period}\cite{Ma}
Given holomorphic $1$-form $\mathrm{d}h$ and meromorphic functions $\phi,\psi:M\rightarrow \mathbb{C}\cup\{\infty\}$
globally defined on a Riemann surface $M$. Suppose they satisfy the regularity condition
1),2) and period conditions 3) as below:

1) $\phi\neq\bar{\psi}$ on $M$ and their poles do not coincide;

2) The zeros of $\mathrm{d}h$ coincide with the poles of $\phi$ or $\psi$
with the same order;

3) Along any closed path the periods satisfy
\begin{equation}\label{eq-period1}
\oint_\gamma \phi \mathrm{d}h
=-\overline{\oint_\gamma \psi \mathrm{d}h }, ~~~(\text{horizontal period condition})
\end{equation}
\begin{equation}\label{eq-period2}
\mathrm{Re}\oint_\gamma \mathrm{d}h=\mathrm{Re}\oint_\gamma \phi\psi \mathrm{d}h=0.~~~(\text{vertical period condition})
\end{equation}
Then
\eqref{x} defines a stationary surface ${\bf x}:M\rightarrow \mathbb{R}^4_1$.

Conversely, any stationary surface ${\bf x}:M\rightarrow \mathbb{R}^4_1$ can be
represented as \eqref{x} in terms of such $\phi,\ \psi$ and $\mathrm{d}h$ over a (necessarily non-compact) Riemann surface $M$.
\end{theorem}

The structure equations and the integrability conditions are given
in \cite{Ma}. An important formula
gives the total Gaussian and normal curvature:
\begin{equation}\label{eq-totalcurvature}
\int_M(-K+\mathrm{i}K^{\perp})\mathrm{d}M=
2\mathrm{i}\int_M
\frac{\phi_z\bar{\psi}_{\bar{z}}}{(\phi-\bar{\psi})^2}
\mathrm{d}z\wedge \mathrm{d}\bar{z}.
\end{equation}
At one end $p$ with $\phi=\bar\psi$, the integral of
total curvature above will become an improper integral.
A crucial observation in \cite{Ma} is that this improper
 integral converges absolutely only for a special class of
such ends.

\begin{definition}\label{def-singular}
Suppose ${\bf x}:D-\{0\}\rightarrow \mathbb{R}^4_1$
is an annular end of a regular stationary surface (with boundary)
whose Gauss maps $\phi$ and $\psi$
extend to meromorphic functions on the unit disk $D\subset\mathbb{C}$.
It is called a \emph{regular end} when
\[
\phi(0)\ne\bar\psi(0).~~~(\text{Thus}~ \phi(z)\ne\bar\psi(z),~\forall~z\in D.)
\]
It is a \emph{singular end} if $\phi(0)=\bar\psi(0)$
where the value could be finite or $\infty$.

When the multiplicities of $\phi$ and $\bar\psi$ at $z=0$
are equal, we call $z=0$ a \emph{bad singular end}.
Otherwise it is a \emph{good singular end}.
\end{definition}

\begin{proposition}\cite{Ma}\label{prop-goodsingular}
A singular end of a stationary surface ${\bf x}:D-\{0\}\rightarrow \mathbb{R}^4_1$
is \emph{good} if and only if the curvature integral
\eqref{eq-totalcurvature} converges absolutely around this end.
\end{proposition}

For a good singular end we introduced the following
definition of its index.

\begin{definition}\cite{Ma}\label{lemma-index2}
Suppose $p$ is an isolated zero of $\phi-\bar\psi$ in $p$'s neighborhood $D_p$,
where holomorphic functions $\phi$ and $\psi$ take the value
$\phi(p)=\overline{\psi(p)}$ with multiplicity $m$ and $n$, respectively.
\emph{The index of $\phi-\bar{\psi}$ at $p$} (when $\phi,\psi$ are both holomorphic at $p$) is
\begin{equation}\label{eqind}
\mathrm{ind}_p(\phi-\bar{\psi})\triangleq
\frac{1}{2\pi\mathrm{i}}\oint_{\partial D_{p}}d\ln(\phi-\bar{\psi})=\left\{
   \begin{array}{ll}
          m, & \hbox{$m<n$;} \\
         -n, & \hbox{ $m>n$.}
   \end{array}
 \right.
\end{equation}
\emph{The absolute index of $\phi-\bar{\psi}$ at $p$} is
\begin{equation}\label{eqind+}
\mathrm{ind}^{+}_p(\phi-\bar{\psi})\triangleq
\left|\mathrm{ind}_p(\phi-\bar{\psi})\right|.
\end{equation}
For a regular end our index is still meaningful with
$\mathrm{ind}=\mathrm{ind}^+=0.$ For convenience we also introduce
\begin{equation}\label{eqind10}
\mathrm{ind}^{1,0}\!\triangleq \frac{1}{2}(\mathrm{ind}^{+}\!+\mathrm{ind}),~~~
\mathrm{ind}^{0,1}\!\triangleq \frac{1}{2}(\mathrm{ind}^{+}\!-\mathrm{ind}),
\end{equation}
which are always non-negative.
\end{definition}

\medskip
A stationary surface in $\mathbb{R}^4_1$ is called an
\emph{algebraic stationary surface} if there exists a
compact Riemann surface $\overline{M}$
with $M=\overline{M}\backslash\{p_1,p_2,\cdots,p_r\}$ such that
${\bf x}_z \mathrm{d}z$ is a vector valued meromorphic form defined
on $\overline{M}$. For this surface class, we have
index formulas, and the Jorge-Meeks formula \cite{Jor-Meeks} is generalized as below.

\begin{theorem}\cite{Ma}\label{GB2}
For a complete algebraic stationary surface
${\bf x}:M\rightarrow \mathbb{R}^4_1$
given by \eqref{x} in terms of $\phi,\psi,\mathrm{d}h$
without bad singular ends, the total Gaussian curvature and total normal curvature are related with the indices at the ends $p_j$ (singular or regular) by the following formulas:
\begin{align}
\int_M K\mathrm{d}M &=-4\pi \left(\deg\phi-\sum{_j} \mathrm{ind}^{1,0}(\phi-\bar{\psi})\right) \label{eq-deg1}\\
&=-4\pi \left(\deg\psi-\sum{_j} \mathrm{ind}^{0,1}(\phi-\bar{\psi})\right),\label{eq-deg2}\\
&=2\pi\left(2-2g-r-\sum_{j=1}^r \widetilde{d}_j\right),
\end{align}
where $g$ is the genus of the compact Riemann surface $\overline{M}$, $r$ is the number of ends, $M\cong \overline{M}\backslash\{p_1,\cdots,p_r\}$,
 and $\widetilde{d}_j$ is the multiplicity of $p_j$
defined to be
\[
\widetilde{d}_j=d_j-\mathrm{ind}^+_{p_j},
\]
where $d_j+1$ is equal to the order of the pole of~
 ${\bf x}_z \mathrm{d}z$ at $p_j$.
\end{theorem}

\begin{proposition}\cite{Ma}\label{ineq-multiplicity}
Let ${\bf x}:D-\{0\}\rightarrow \mathbb{R}^4_1$ be a regular or a good
singular end which is further assumed to be complete at $z=0$. Then
its multiplicity satisfies $\widetilde{d}\ge 1.$
\end{proposition}

For a non-orientable stationary surface in $\mathbb{R}^4_1$,
consider its oriented double covering surface. This allows
us to describe it using Weierstrass representation formula.
Indeed this is a direct generalization of Meeks' result in \cite{Meeks}.

\begin{theorem}\cite{MaWang}\label{thm-nonorientable}
Let $\widetilde{M}$ be a Riemann surface with an anti-holomorphic involution
$I:\widetilde{M}\to \widetilde{M}$ (i.e., a conformal automorphism of
$\widetilde{M}$ reversing the orientation) without fixed points.
Let $\{\phi,\psi,\mathrm{d}h\}$ be a set of Weierstrass data on $\widetilde{M}$
such that
\begin{equation}\label{eq-nonorientable}
\phi\circ I=\bar\psi,~~
\psi\circ I=\bar\phi,~~
I^*\mathrm{d}h=\overline{\mathrm{d}h},
\end{equation}
which satisfy the regularity and period conditions as well.
Then they determine a non-orientable stationary surface
\[
M=\widetilde{M}/\{\mathrm{id},I\}\to \mathbb{R}^4_1
\]
by the Weierstrass representation formula \eqref{x}.

Conversely, any non-orientable stationary surface ${\bf x}:M\to \mathbb{R}^4_1$
could be constructed in this way.
\end{theorem}

\section{Existence result}

Before stating our existence result, we first make an analysis
of possible examples with $-\int K\mathrm{d}M =6\pi$.
By the generalized Jorge-Meeks formula,
\[
-\int_M K\mathrm{d}M=6\pi=2\pi\Big(g+r-1+\sum_{j=1}^{r} \tilde{d}_j\Big)~.
\]
Thus $g+r+\sum_{j=1}^{r} \tilde{d}_j=4$.
The lower bound estimation \eqref{eq-bound} has ruled out the possibility of
$g\ge 1$. So $g=0$, and we are left with two cases:
one end with multiplicity $\tilde{d}=3$, or two ends with
multiplicities $\tilde{d}_1=\tilde{d}_2=1$.

In the first case where $M$ is a M\"obius strip,
we have Meeks' example and its generalization,
whose unique end is regular. It was unknown to us during the
work of \cite{MaWang} whether there exist more examples with
regular ends or good singular ends.

In the proposition below,
we not only confirm the existence of more examples,
but also construct the first non-orientable stationary surface
with a good singular end.

\begin{proposition}
There exists a one-parameter family of complete, immersed, stationary M\"obius strip
in $\mathbb{R}^4_1$ with a good singular end and $-\int K\mathrm{d}M =6\pi$.
\end{proposition}
\begin{proof}
Let the oriented double covering surface be
$\widetilde{M}=\mathbb{C}\backslash\{0\}$ with two good singular
ends $0,\infty$, and the orientation reversing involution
without fixed points be $I:z\mapsto -1/\bar{z}$.

Without loss of generality,
suppose the good singular end $z=0$ has $\mathrm{ind}=m$
and both $\phi,\psi$ take the value $0$.
Then $\phi(0)=0$ with multiplicity $m$,
$\psi(0)=0$ with multiplicity greater than $m$.
On the other hand, the index formula together with
$-\int_M K\mathrm{d}M=6\pi$ implies $\deg\phi=m+3$.
Thus the meromorphic $\phi$ over the Riemann sphere must take
the form
\[
\phi(z)=\frac{z^m(b_3 z^3+ b_2 z^2+ b_1 z+ b_0)}
{a_{m+3}z^{m+3}+\cdots+ a_1 z+a_0},
\]
and
\[
\psi(z)=\overline{\phi(-1/\bar{z})}=
\frac{-\bar{b}_3+\bar{b}_2z-\bar{b}_1z^2+\bar{b}_0z^3}
{(-1)^{m+3}\bar{a}_{m+3}+\cdots+\bar{a}_0z^{m+3}}.
\]
Since $\psi(0)=0$, we know $b_3=0$. Since $\deg\phi=m+3$,
there must be $a_{m+3}\ne 0$. Finally, $\psi(0)=0$ with a
multiplicity at most $3$, so $m=1$ or $2$.

If there exists an example as above with $m=2$, then $\psi(0)=0$
with multiplicity $3$, thus $b_1=b_2=b_3=0$.
$b_0$ has to be nonzero, and one may assume $b_0=1$
without loss of generality.
As the consequence of $\deg\phi=m+3=5$, we obtain
\[
\phi(z)=\frac{z^2}{a_5z^5+\cdots+ a_1 z+a_0}, ~~
\psi(z)=\frac{z^3}{-\bar{a}_5+\cdots+\bar{a}_0z^5}.
\]
Since $\mathrm{d}h$ is holomorphic on $\widetilde{M}=\mathbb{C}\backslash\{0\}$,
and $I^*\mathrm{d}h=\overline{\mathrm{d}h}$, up to a real non-zero constant it takes the form
\[
\mathrm{d}h=\mathrm{i}\frac{(a_5z^5+\cdots+ a_1 z+a_0)(-\bar{a}_5+\cdots+\bar{a}_0z^5)}{z^6}\mathrm{d}z.
\]
But $\phi\psi\mathrm{d}h=\mathrm{i}\frac{\mathrm{d}z}{z}$
has residue $-2\pi$, which violates the (vertical) period condition.

Thus the possible example with a good singular end must have
$\mathrm{ind}_0=m=1$ and $\deg\phi=4$. $\psi(0)=0$ with a multiplicity
at least $2$, so $b_3=b_2=0$. Now we may write
\begin{equation}\label{eq-example1}
\phi(z)=\frac{z(b_1 z+ b_0)}{a_4z^4+\cdots+ a_1 z+a_0},~~
\psi(z)=\frac{-\bar{b}_1z^2+\bar{b}_0z^3}
{\bar{a}_4+\cdots+\bar{a}_0z^4}.
\end{equation}
By the similar reason as in the paragraph above, up to a real factor we have
\begin{equation}\label{eq-example2}
\mathrm{d}h=\mathrm{i}\frac{(a_4z^4+\cdots+ a_1 z+a_0)(\bar{a}_4+\cdots+\bar{a}_0z^4)}{z^5}\mathrm{d}z.
\end{equation}
Taking integral along a simple closed path around $z=0$. The period conditions imply
\begin{eqnarray}
\mathrm{Re}\oint \mathrm{d}h=0 ~~&\Rightarrow&~~
|a_4|^2-|a_3|^2+|a_2|^2-|a_1|^2+|a_0|^2=0,\label{eq-period1}\\
\mathrm{Re}\oint \phi\psi\mathrm{d}h=0 ~~&\Rightarrow&~~
|b_1|^2-|b_0|^2=0,\label{eq-period2}\\
\oint \phi\mathrm{d}h+\overline{\oint \psi\mathrm{d}h}=0
~~&\Rightarrow&~~ b_1\bar{a}_2-b_0\bar{a}_1=0.\label{eq-period3}
\end{eqnarray}
Now we choose the parameters
\begin{equation}\label{eq-parameter}
b_0=-1,b_1=1,a_1=a_2=0,a_0=\sqrt{2\epsilon+\epsilon^2},
a_3=1+\epsilon,a_4=1,
\end{equation}
which satisfy the period conditions \eqref{eq-period1}\eqref{eq-period2}\eqref{eq-period3}.
We need only to verify that $\phi(z)\ne\bar\psi(z)$
for the corresponding functions given by \eqref{eq-example1}
and parameters chosen above.

We assert that when $\epsilon\in\mathbb{R}$ is a sufficiently
small positive number,
\[\phi(z)=\bar\psi(z)=\phi(-1/\bar{z})\]
has no solutions except the trivial ones $z=0$ (and $z=\infty$).
Otherwise, suppose there is a non-zero solution
$z_0=r\mathrm{e}^{\mathrm{i}\theta}\in\mathbb{C}\backslash\{0\}$
with radial $r>0$ and argument $\theta$ such that
\[\phi(z_0)=\phi(-1/\bar{z}_0)=t\]
for some complex value $t$. (Note that when $z_0\ne 0$,
$t$ must be nonzero. Otherwise there is no such solution $z_0$.)
Invoking \eqref{eq-example1}\eqref{eq-parameter}, we may rewrite
\begin{equation}\label{eq-z0}
\phi(z)=t ~~\Leftrightarrow~~
z^4+(1+\epsilon)z^3-\frac{1}{t}(z^2-z)+\sqrt{2\epsilon+\epsilon^2}=0,
\end{equation}
which have two roots $z_0,-1/\bar{z}_0$ and other two $z_1,z_2$.
By Vi\`{e}ta's formula,
\begin{eqnarray}
  z_0-\frac{1}{\bar{z}_0}+z_1+z_2 &=& -(1+\epsilon), \label{eq-vieta1} \\
  -\frac{z_0}{\bar{z}_0}+z_1 z_2 +(z_0-\frac{1}{\bar{z}_0})(z_1+z_2) &=& -\frac{1}{t},\label{eq-vieta2}\\
  -\frac{z_0}{\bar{z}_0}(z_1+z_2)+z_1 z_2 (z_0-\frac{1}{\bar{z}_0}) &=& -\frac{1}{t}, \label{eq-vieta3}\\
  -\frac{z_0}{\bar{z}_0}\cdot z_1 z_2 &=& \sqrt{2\epsilon+\epsilon^2}.\label{eq-vieta4}
\end{eqnarray}
Express $z_1+z_2$ and $z_1z_2$ in terms of
$z_0,-\frac{1}{\bar{z}_0}$ using
\eqref{eq-vieta1}\eqref{eq-vieta4}.
Since $z_0=r\mathrm{e}^{\mathrm{i}\theta}$,
$-\frac{z_0}{\bar{z}_0}=-\mathrm{e}^{2\mathrm{i}\theta}$,
$z_0-\frac{1}{\bar{z}_0}=(r-\frac{1}{r})\mathrm{e}^{\mathrm{i}\theta}$.
Substitute these into \eqref{eq-vieta2}\eqref{eq-vieta3}
and eliminate $\frac{1}{t}$. Collecting terms we get
\begin{eqnarray*}
  0 &=& \mathrm{e}^{2\mathrm{i}\theta}
  \left[(r-\frac{1}{r})^2+2(r-\frac{1}{r})\cos\theta+2\right] \\
  && +\epsilon\cdot\mathrm{e}^{\mathrm{i}\theta}
  (r-\frac{1}{r}+\mathrm{e}^{\mathrm{i}\theta})  -\sqrt{2\epsilon+\epsilon^2}\mathrm{e}^{-\mathrm{i}\theta}
  (r-\frac{1}{r}-\mathrm{e}^{-\mathrm{i}\theta})\\
  &=& \mathrm{e}^{2\mathrm{i}\theta}(|w|^2+1)
  +\epsilon \cdot\mathrm{e}^{\mathrm{i}\theta} w
  -\sqrt{2\epsilon+\epsilon^2} \mathrm{e}^{-\mathrm{i}\theta}(w-2\cos\theta)
\end{eqnarray*}
with $w=r-\frac{1}{r}+\mathrm{e}^{\mathrm{i}\theta}$.
Because $(|w|^2+1)/(|w|+1)$ has a positive lower bound,
it is easy to see that when $\epsilon$ is small enough the equality could not hold true for any $r$ and $\theta$.
This verifies $\phi(z)\ne\bar\psi(z)$ for any $z\ne 0$
and finishes the proof.
\end{proof}

\begin{remark}
An interesting observation is that when $\epsilon=0$,
the Gauss maps and the height differential given by
\eqref{eq-example1}\eqref{eq-example2} degenerate to
\[
\phi=\frac{z-1}{z^2(z+1)}=-1/\psi, ~~\mathrm{d}h=\mathrm{i}\frac{(z+1)(z-1)}{z^2}\mathrm{d}z,
\]
which are exactly the Weierstrass data of Meeks' M\"obius strip
in $\mathbb{R}^3$ \cite{Meeks} (up to a change of coordinate),
which satisfies the regularity condition $\phi\ne\bar\psi$ obviously.
This also provides an example of a family of
stationary surfaces with a good singular end which having
its limit as a regular end of a stationary surface.
\end{remark}

In the end of this section, we briefly give the description
of the known examples with $-\int K\mathrm{d}M=6\pi$.

\begin{example}
[Generalization in $\mathbb{R}^4_1$ of Meeks' minimal M\"obius strip \cite{MaWang}]
\label{exa-meeks}
This is defined on $\widetilde{M}=\mathbb{C}\backslash\{0\}$ with involution $I:z\to -1/\bar{z}$,
and the Weierstrass data be
\begin{equation}
\phi=\frac{z-\lambda}{z-\bar\lambda}\cdot z^{2m},~~
\psi=\frac{1+\bar\lambda z}{1+\lambda z}\cdot\frac{1}{z^{2m}},~~
\mathrm{d}h=\mathrm{i}\frac{(z-\bar\lambda)(1+\lambda z)}{z^2}\mathrm{d}z,
\end{equation}
where $\lambda$ is a complex parameter satisfing
$\lambda\ne \pm 1, |\lambda|=1$, and the integer $m\ge 1$.
When $m=1$ we have a one-parameter family of examples with total curvature $6\pi$ .
\end{example}
When $\lambda=\pm\mathrm{i}$ we have $\phi=-1/\psi$, and the example above
is equivalent to Oliveira's examples in $\mathbb{R}^3$
 \cite{Oliveira}.
(Meeks' example \cite{Meeks} corresponds to the case $m=1,\lambda=\mathrm{i}$.)

It is not difficult to imagine that Meeks' example in $\mathbb{R}^3$ allows
a larger family of deformations depending on more parameters,
and the moduli space has a higher dimension.
The rough idea behind this belief is that the regularity condition $\phi\ne \bar\psi$ is an open property of the parameter space. Yet we are not interested in deciding this moduli space, even the dimension.

Below is another interesting family with an essential singularity in their Weierstrass data at the end.

\begin{example}
[stationary M\"obius strips with essential singularities and finite total curvature \cite{MaWang}]
\label{exa-essen2}
\begin{equation}\label{ex-essential2}
\phi=z^{2p-1}\mathrm{e}^{\frac{1}{2}(z-\frac{1}{z})},~~
\psi=\frac{-1}{z^{2p-1}}\mathrm{e}^{\frac{1}{2}(z-\frac{1}{z})},~~
\mathrm{d}h=\mathrm{d}~\mathrm{e}^{-\frac{1}{2}(z-\frac{1}{z})}.~~(p\in\mathbb{Z}_{\ge 2})
\end{equation}
These are complete immersed stationary M\"obius strip with finite total curvature
\[
-\int_M K\mathrm{d}M=2(2p-1)\pi,~~\int_M K^{\perp}\mathrm{d}M=0.
\]
In particular, when $p=2$ we have an example with total Gaussian curvature $6\pi$.
\end{example}

\section{Non-existence of examples with two ends}

This part is devoted to the most involved proof of our conclusion 2 that the topological type of the projective plane with two punctures never occur. We make an interesting observation at the beginning.

\begin{lemma}\label{lem-flux}
Let ${\bf x}:M\to \mathbb{R}^4_1$ be an algebraic stationary surface which is homomorphic to a M\"obius strip with several punctures. Then its period at any end must vanish.
\end{lemma}
\begin{proof}
Without loss of generality, let $\widetilde{M}=\mathbb{C}\backslash\{0,c,-1/\bar{c},\cdots\}$
be the oriented double covering of $M=\widetilde{M}/\{I,id\}$ with
orientation-reversing involution $I(z)=-1/\bar{z}$, and $z=0$ corresponds to the end in concern.

By assumption, ${\bf x}_z \mathrm{d}z$ is a vector-valued meromorphic 1-form globally defined on
$\mathbb{C}\cup\{\infty\}$. So it can has a unique partial fraction decomposition. Because
$z=0$ is a pole of ${\bf x}_z \mathrm{d}z$, this sum contains the term $v_1\frac{\mathrm{d}z}{z}$.
The period condition shows that $v_1\in \mathbb{R}^4_1$ is a real vector.

On the other hand, Theorem~\ref{thm-nonorientable} implies that this
vector-valued 1-form, or equivalently, the sum of these partial fractions, satisfies
\[
I^*({\bf x}_z \mathrm{d}z)=\overline{{\bf x}_z \mathrm{d}z}.
\]
Since $I(z)=-1/\bar{z}$, by comparing the coefficients of the term $\mathrm{d}\bar{z}/\bar{z}$ at both sides, we see $v_1=-v_1$.
So the period vector $v_1$ must vanish.
\end{proof}

\begin{remark}
We do not know whether our result still holds true for non-oriented stationary surfaces in $\mathbb{R}^4_1$ with higher genus.
Note that the proof above relies on three facts: the existence
of the global coordinate $z$, the explicit form of
$I:z\to -1/\bar{z}$, and the expression of a meromorphic function as a sum of partial fractions. All of them depend on the assumption
that the compactification of the oriented double covering is the Riemann sphere.
\end{remark}

To show the non-existence of complete, immersed stationary surfaces with total curvature $-\int K\mathrm{d}M=6\pi$ and two ends, we use proof by contradiction.
Suppose there is such one ${\bf x}:M\to \mathbb{R}^4_1$. Then it must be non-oriented, whose
double covering space $\widetilde{M}$ is the 2-sphere punctured at two pairs of antipodal points (which must be located on a
unique circle).
The antipodal map could be represented by the involution $I:z\to -1/\bar{z}$ on the Riemann sphere $\mathbb{C}\cup\{\infty\}$,
and the circle corresponds to the real line on $\mathbb{C}$.
Thus the four punctures may be chosen to be $\{0,\infty; c,-1/c\}$ and the parameter $c\in\mathbb{R}-\{0\}$.
\begin{proposition}
There exists no complete, immersed stationary surface ${\bf x}:M\to \mathbb{R}^4_1$ with total curvature $-\int K\mathrm{d}M=6\pi$ and two regular ends.
\end{proposition}
\begin{proof}
On the orientable double covering space
$\widetilde{M}=\mathbb{C}\backslash\{0,c,-1/c\}$,
${\bf x}_z \mathrm{d}z$ has a pole of order 2 without residue
at each regular end $0,\infty, c,-1/c$. As a vector-valued rational function, ${\bf x}_z$ has a decomposition into rational fractions as below
\[
{\bf x}_z=\frac{1}{z^2}{\bf v}_1+\bar{\bf v}_1
+\frac{1}{(z-c)^2}{\bf u}_1
+\frac{1}{(cz+1)^2}\bar{\bf u}_1
\]
according to the restriction
$\overline{{\bf x}_z \mathrm{d}z}=I^*({\bf x}_z \mathrm{d}z)$.

As a well-known consequence of the conformal condition,
${\bf v}_1,{\bf u}_1\in \mathbb{C}^4_1$ are both isotropic vectors,
and the assumption of regular ends implies that the real
and imaginary parts of ${\bf v}_1,{\bf u}_1$ must be
spacelike vectors orthogonal to each other with equal length.
The conformal condition further implies
\begin{eqnarray}
0=\frac{1}{2}{\bf x}_z\cdot {\bf x}_z&=&
\frac{{\bf v}_1{\bf u}_1}{z^2(z-c)^2}
+\frac{{\bf v}_1\bar{\bf u}_1}{z^2(cz+1)^2}
+\frac{{\bf u}_1\bar{\bf u}_1}{(z-c)^2(cz+1)^2}\notag \\
&&~~
+\frac{{\bf v}_1\bar{\bf v}_1}{z^2}
+\frac{\bar{\bf v}_1{\bf u}_1}{(z-c)^2}
+\frac{\bar{\bf v}_1\bar{\bf u}_1}{(cz+1)^2}~.\label{eq-conf1}
\end{eqnarray}
Here terms like ${\bf v}_1{\bf u}_1$ are inner products between
(complex) vectors. Because
\begin{eqnarray}
\frac{1}{z^2(z-c)^2}&=&
\frac{1}{c^2}\left[\frac{1}{z^2}+\frac{1}{(z-c)^2}\right]
+\frac{2}{c^3}\left[\frac{1}{z}-\frac{1}{z-c}\right]~,
\label{eq-frac1}\\
\frac{1}{z^2(cz+1)^2}&=&
\frac{1}{z^2}+\frac{c^2}{(cz+1)^2}
-\frac{2c}{z}+\frac{2c^2}{cz+1}~,
\label{eq-frac2}\\
\frac{1}{(z-c)^2(cz+1)^2}&=&
\frac{1}{(1+c^2)^2}\frac{1}{(z-c)^2}
+\frac{c^2}{(1+c^2)^2}\frac{1}{(cz+1)^2}\notag\\
&&~-\frac{2c}{(1+c^2)^3}\frac{1}{z-c}
+\frac{2c^2}{(1+c^2)^3}\frac{1}{cz+1}~,
\label{eq-frac3}
\end{eqnarray}
the vanishing of the coefficients of the terms like
$\frac{1}{z},\frac{1}{z-c},\cdots,\frac{1}{(cz+1)^2}$
in \eqref{eq-conf1} yields
\begin{eqnarray*}
{\bf v}_1{\bf u}_1\cdot\frac{2}{c^3}
-2c{\bf v}_1\bar{\bf u}_1=0,\\
{\bf v}_1{\bf u}_1\cdot\frac{-2}{c^3}
+|{\bf u}_1|^2\cdot \frac{-2c}{(1+c^2)^3}=0,\\
{\bf v}_1\bar{\bf u}_1\cdot 2c^2
+|{\bf u}_1|^2\cdot \frac{2c^2}{(1+c^2)^3}=0,\\
{\bf v}_1{\bf u}_1\cdot\frac{1}{c^2}
+{\bf v}_1\bar{\bf u}_1+|{\bf v}_1|^2=0,\\
\bar{\bf v}_1{\bf u}_1+{\bf v}_1{\bf u}_1\cdot\frac{1}{c^2}
+|{\bf u}_1|^2\cdot \frac{1}{(1+c^2)^2}=0,\\
\bar{\bf v}_1\bar{\bf u}_1+{\bf v}_1\bar{\bf u}_1\cdot c^2
+|{\bf u}_1|^2\cdot \frac{c^2}{(1+c^2)^2}=0.
\end{eqnarray*}
Without loss of generality, let ${\bf v}_1=
(1,\mathrm{i},0,0)^t$ and $|{\bf v}_1|^2=2$. Then the
equations above must have solution
\[
|{\bf v}_1|^2=2,~|{\bf u}_1|^2=2(1+c^2)^2,~
{\bf v}_1{\bf u}_1=\frac{-2c^4}{1+c^2}~,~
{\bf v}_1\bar{\bf u}_1=\frac{-2}{1+c^2}~.
\]
Let ${\bf u}_1=(\alpha_1+\mathrm{i}\beta_1,\cdots,
\alpha_4+\mathrm{i}\beta_4)^t$ where $\alpha_j,\beta_j\in\mathbb{R}$. Then the equations above imply
\begin{equation}\label{eq-isotropic0}
\beta_1=\alpha_2=0,~~\alpha_1=\frac{-(1+c^4)}{1+c^2}~,~
\beta_2=\frac{c^4-1}{1+c^2}=c^2-1.
\end{equation}
Finally, the real and imaginary part of ${\bf u}_1$ must be
orthogonal to each other with the same length $1+c^2$, i.e.,
\begin{equation}\label{eq-isotropic1}
(\alpha_3)^2-(\alpha_4)^2=(1+c^2)^2-(\alpha_1)^2,~~
(\beta_3)^2-(\beta_4)^2=(1+c^2)^2-(\beta_2)^2,
\end{equation}
and
\begin{equation}\label{eq-isotropic2}
\alpha_3\beta_3=\alpha_4\beta_4.
\end{equation}
From \eqref{eq-isotropic0}, $|\alpha_1|,|\beta_2|<1+c^2$, so \eqref{eq-isotropic1} implies $|\alpha_3|>|\alpha_4|,|\beta_3|>|\beta_4|$.
But in this situation \eqref{eq-isotropic2} will never be true.
This contradiction finishes our proof.
\end{proof}

To treat examples with good singular ends, we need the following lemma which gives a normalized Laurent expansion of ${\bf x}_z$.

\begin{lemma}\label{lem-singular}
On a complete, algebraic stationary surface ${\bf x}:M\to \mathbb{R}^4_1$,
suppose there is a local complex coordinate $z$
with a good singular end at $z=0$ which has $\tilde{d}=1, \mathrm{ind}=m\ge 1$, and the flux vector is zero
(i.e., it has even no imaginary period at this end).
Then the vector-valued meromorphic function ${\bf x}_z$ has a Laurent series around $z=0$ as below:
\begin{equation}\label{eq-singular}
{\bf x}_z=\left(\sum_{k=1}^m \frac{\alpha_{k+2}}{z^{k+2}}\right)
{\bf v}_0 +\frac{1}{z^2}{\bf v}_1 + O(1),
\end{equation}
where $\alpha_{m+2}\ne 0$,
${\bf v}_0\in \mathbb{R}^4_1$ is a non-zero lightlike vector,
${\bf v}_1\in \mathbb{C}^4_1$ is an isotropic vector,
and ${\bf v}_0,\mathrm{Re}({\bf v}_1),\mathrm{Im}({\bf v}_1)$
span a degenerate 3-dimensional subspace.
\end{lemma}
\begin{proof}
Without loss of generality we may suppose the Gauss maps
$\phi(0)=\psi(0)=0$.
The end $z=0$ has multiplicity $d=m+\tilde{d}=m+1$.
As a consequence, ${\bf x}_z \mathrm{d}z$ must have a pole of order $m+2$, which is the same as $\mathrm{d}h$.
One need only to verify the conclusion with respect to a
specific coordinate $z$, which we choose to make
\[
\mathrm{d}h=\frac{\mathrm{d}z}{z^{m+2}},~~
\phi(z)=a_m z^m +a_{m+1} z^{m+1} +o(|z|^{m+1}),~~
\phi(z)=b_{m+1} z^{m+1} +o(|z|^{m+1}).
\]
Then by the Weierstrass representation formula we have
\begin{eqnarray*}
{\bf x}_z \mathrm{d}z &=& (\phi+\psi, -\mathrm{i}(\phi-\psi),1-\phi\psi,1+\phi\psi)\mathrm{d}h\\
&=&
\frac{\mathrm{d}z}{z^{m+2}}
\begin{pmatrix} 0\\ 0\\ 1\\ 1\end{pmatrix}
+\frac{\mathrm{d}z}{z^2}
\begin{pmatrix} a_m\\ -\mathrm{i} a_m\\ 0\\ 0\end{pmatrix}
+\frac{\mathrm{d}z}{z}
\begin{pmatrix} a_{m+1}+b_{m+1}\\ -\mathrm{i} (a_{m+1}-b_{m+1})\\ 0\\ 0\end{pmatrix}+O(1).
\end{eqnarray*}
As we assumed, the coefficient vector of the term $\frac{\mathrm{d}z}{z}$ must vanish, and the conclusion follows
immediately.
\end{proof}

After these preparation, we deal with the second case.

\begin{proposition}
There exists no complete, immersed stationary surface ${\bf x}:M\to \mathbb{R}^4_1$ with total curvature $-\int K\mathrm{d}M=6\pi$ and two good singular ends.
\end{proposition}
\begin{proof}
As before we consider the orientable double covering
$\widetilde{M}=\mathbb{C}\backslash\{0,c,-1/c\}$ with involution $I:z\to -1/\bar{z}$ where each end has $\tilde{d}=1$.
Here $z$ is the global coordinate on $\mathbb{C}$.

Without loss of generality, assume that at the pair of ends $z=0,\infty$ we have $\phi(0)=\psi(0)=0$, and $z=0$ has index $\mathrm{ind}=m\ge 1$. $z=c,-1/c$ is another pair of
singular ends with $\mathrm{ind}=n\ge 1$.
As a vector-valued rational function,
by Lemma~\ref{lem-flux} and Lemma~\ref{lem-singular}
${\bf x}_z$ has a decomposition as below:
\begin{eqnarray}
{\bf x}_z&=&\sum_{k=1}^m \left(\frac{\alpha_{k+2}}{z^{k+2}}
+(-z)^k\bar\alpha_{k+2}\right){\bf v}_0
+\frac{1}{z^2}{\bf v}_1 + \bar{\bf v}_1\notag \\
&&\!\!+\sum_{j=1}^n \left(\frac{\beta_{j+2}}{(z-c)^{j+2}}
+\frac{\bar\beta_{j+2}(-z)^j}{(cz+1)^{j+2}}\right){\bf u}_0
+\frac{1}{(z-c)^2}{\bf u}_1 + \frac{1}{(cz+1)^2}\bar{\bf u}_1.
\label{eq-decom1}
\end{eqnarray}
where we have used the condition
$\overline{{\bf x}_z \mathrm{d}z}=I^*({\bf x}_z \mathrm{d}z)$.
Notice that the 3-space
\[
V^3=\mathrm{Span}\{{\bf v}_0,\mathrm{Re}({\bf v}_1),\mathrm{Im}({\bf v}_1)\}=\{{\bf v}_0\}^\bot
\]
and
\[
U^3=\mathrm{Span}\{{\bf u}_0,\mathrm{Re}({\bf u}_1),\mathrm{Im}({\bf u}_1)\}=\{{\bf u}_0\}^\bot.
\]
are both degenerate 3-spaces isometric to $\mathbb{R}^3_0$.

We claim ${\bf v}_0$ is not parallel to ${\bf u}_0$.
Otherwise, ${\bf v}_0$ will be orthogonal to all of
${\bf v}_0,{\bf v}_1,\bar{\bf v}_1,{\bf u}_1,\bar{\bf u}_1$,
and $<{\bf x},{\bf v}_0>$ is a harmonic function defined on
the whole compact Riemann surface. This implies $<{\bf x},{\bf v}_0>$ is a constant; in other words, our surface is located in
an affine $3$-space orthogonal to ${\bf v}_0$. Such a surface
has flat metric and total curvature $0$, which is a contradiction.

For this reason, we may suppose
\[
{\bf v}_0=\begin{pmatrix} 0\\ 0\\ 1\\ 1\end{pmatrix},~
{\bf u}_0=\begin{pmatrix} 0\\ 0\\ -1\\ 1\end{pmatrix},~~~\Rightarrow~~
V^3\cap U^3=\{{\bf v}_0,{\bf u}_0\}^\bot
=\Big\{\begin{pmatrix} a\\ b\\ 0\\ 0\end{pmatrix}|~\forall~ a,b\in \mathbb{R}\Big\}~.
\]
Without loss of generality we suppose ${\bf v}_1$
is a linear combination of ${\bf v}_0$ and ${\bf w}=(1,\mathrm{i},0,0)^t$, and ${\bf u}_1$
a linear combination of ${\bf u}_0$ and ${\bf w}$.
This allows us to rewrite
\begin{eqnarray*}
{\bf x}_z&=&\sum_{k=0}^m \left(\frac{\alpha_{k+2}}{z^{k+2}}
+(-z)^k\bar\alpha_{k+2}\right){\bf v}_0
+\frac{\gamma}{z^2}{\bf w} + \bar\gamma\bar{\bf w}\\
&&\!\!+\sum_{j=0}^n \left(\frac{\beta_{j+2}}{(z-c)^{j+2}}
+\frac{\bar\beta_{j+2}(-z)^j}{(cz+1)^{j+2}}\right){\bf u}_0
+\frac{\delta}{(z-c)^2}{\bf w} + \frac{\bar\delta}{(cz+1)^2}
\bar{\bf w}.
\end{eqnarray*}
The conformal condition implies
\begin{eqnarray*}
0=\frac{1}{2}{\bf x}_z\cdot {\bf x}_z&=&
\sum_{k=0}^m \left(\frac{\alpha_{k+2}}{z^{k+2}}
+(-z)^k\bar\alpha_{k+2}\right)\cdot
\sum_{j=0}^n \left(\frac{\beta_{j+2}}{(z-c)^{j+2}}
+\frac{\bar\beta_{j+2}(-z)^j}{(cz+1)^{j+2}}\right)
{\bf v}_0{\bf u}_0\\
&&+\frac{\gamma \bar\delta}{z^2(cz+1)^2}|{\bf w}|^2
+\frac{\delta \bar\gamma}{(z-c)^2}|{\bf w}|^2.
\end{eqnarray*}
Since $\alpha_{m+2}\ne 0$,${\bf v}_0{\bf u}_0=-2$,
the first term has a pole of order $m+2\ge 3$ at $z=0$,
which could not be canceled by the second and the third term.
This contradiction finishes the proof.
\end{proof}
\begin{theorem}
There exists no complete, immersed stationary surface ${\bf x}:M\to \mathbb{R}^4_1$ with total curvature $-\int K\mathrm{d}M=6\pi$ and two ends.
\end{theorem}
\begin{proof}
According to the previous discussion, now we need only to rule
out the final possibility of one regular end and one good singular end. Suppose there is such an example,
which lifts to a stationary immersion ${\bf x}:\widetilde{M}\to \mathbb{R}^4_1$ with the orientable double covering
\[
\widetilde{M}=\mathbb{C}\backslash\{0,c,-1/c\}, ~~I:z\to -1/\bar{z}.
\]
Suppose $z=0$ is a singular end with $\mathrm{ind}_0=m\ge 1$.
Let $z=c,-1/c$ be the regular ends with $c\ne 0, c\in \mathbb{R}$. Similar to \eqref{eq-decom1},
we have the following decomposition
\begin{equation}\label{eq-decom2}
{\bf x}_z=\sum_{k=1}^m \left(\frac{\alpha_{k+2}}{z^{k+2}}
+(-z)^k\bar\alpha_{k+2}\right){\bf v}_0
+\frac{1}{z^2}{\bf v}_1 + \bar{\bf v}_1
+\frac{1}{(z-c)^2}{\bf u}_1 + \frac{1}{(cz+1)^2}\bar{\bf u}_1
\end{equation}
where ${\bf v}_0$ is real and lightlike, and
${\bf v}_1,{\bf u}_1$ are complex isotropic vectors.

We assert: 1) $c=\pm 1$; 2) $m=1$.

Using the conformal condition
${\bf x}_z\cdot {\bf x}_z=0$ and comparing the coefficients of
terms involving $1/z^{m+2}$, we get
\[
{\bf v}_0{\bf u}_1+\frac{1}{c^2}{\bf v}_0\bar{\bf u}_1=0.
\]
If $c\ne \pm 1$, then the equation above together with
its complex conjugation yields
${\bf v}_0{\bf u}_1={\bf v}_0\bar{\bf u}_1=0$.
Thus ${\bf x}_z\mathrm{d}z \cdot {\bf v}_0\equiv 0$.
By maximal principle for harmonic functions,
${\bf x}$ maps the surface into $3$-space $\{{\bf v}_0\}^\bot$
(up to a translation). This is impossible. Thus we may take
$c=1$.

To verify the second assertion, suppose on the contrary we have
$m\ge 2$. Then the terms like $1/z^2(z\pm 1)^2,1/z^2(z\pm 1)$
will not influence the coefficients of $1/z^{m+2}$ and $1/z^{m+1}$
when computing ${\bf x}_z\cdot {\bf x}_z=0$. Because
\[
\frac{1}{z^{m+2}(z\mp 1)^2}=\frac{1}{z^{m+2}}\pm \frac{2}{z^{m+1}}+\cdots,~~
\frac{1}{z^{m+1}(z\mp 1)^2}=\frac{1}{z^{m+1}}+\cdots,
\]
the vanishing of the term $1/z^{m+2}$ implies
${\bf v}_0{\bf u}_1+{\bf v}_0\bar{\bf u}_1=0$,
and the vanishing of the term $1/z^{m+2}$ implies
${\bf v}_0{\bf u}_1-{\bf v}_0\bar{\bf u}_1=0$.
Again this induces ${\bf v}_0{\bf u}_1=0$ and
${\bf x}_z\mathrm{d}z \cdot {\bf v}_0\equiv 0$, which is not allowed. This shows $m=1$.

Without loss of generality we may suppose the non-zero
coefficient $a_{m+2}=a_3$ is unit, and up to a rotation
$z\to \mathrm{e}^{\mathrm{i}\theta}$ we make it to be $1$
(note that this change of coordinate will not influence anything,
including the explicit form of $I(z)=-1/\bar{z}$).
Now we can rewrite \eqref{eq-decom2} as
\begin{equation}\label{eq-decom3}
{\bf x}_z= \left(\frac{1}{z^3}-z\right){\bf v}_0
+\frac{1}{z^2}{\bf v}_1 + \bar{\bf v}_1
+\frac{1}{(z-1)^2}{\bf u}_1 + \frac{1}{(z+1)^2}\bar{\bf u}_1.
\end{equation}
The conformal condition reads
\begin{eqnarray}
0=\frac{1}{2}{\bf x}_z\cdot {\bf x}_z&=&
\left(\frac{1}{z^3}-z\right)\frac{{\bf v}_0{\bf u}_1}{(z-1)^2}
+\left(\frac{1}{z^3}-z\right)\frac{{\bf v}_0\bar{\bf u}_1}{(z+1)^2}
+\frac{{\bf v}_1\bar{\bf v}_1}{z^2}
+\frac{{\bf u}_1\bar{\bf u}_1}{(z^2-1)^2}\notag\\
&&+\frac{{\bf v}_1{\bf u}_1}{z^2(z-1)^2}
+\frac{{\bf v}_1\bar{\bf u}_1}{z^2(z+1)^2}
+\frac{\bar{\bf v}_1{\bf u}_1}{(z-1)^2}
+\frac{\bar{\bf v}_1\bar{\bf u}_1}{(z+1)^2}. \label{eq-conf2}
\end{eqnarray}
Because
\begin{eqnarray*}
\left(\frac{1}{z^3}-z\right)\frac{1}{(z-1)^2}
&=&\frac{1}{z^3}+\frac{2}{z^2}+\frac{3}{z}-\frac{4}{z-1},\\
\left(\frac{1}{z^3}-z\right)\frac{1}{(z+1)^2}
&=&\frac{1}{z^3}-\frac{2}{z^2}+\frac{3}{z}-\frac{4}{z+1},
\end{eqnarray*}
together with \eqref{eq-frac1}, \eqref{eq-frac2}, \eqref{eq-frac3}
when $c=1$, we obtain the following equations from \eqref{eq-conf2} by comparing coefficients of the corresponding
terms:
\begin{eqnarray*}
\frac{1}{z^3}: &&{\bf v}_0{\bf u}_1+{\bf v}_0\bar{\bf u}_1=0,\\
\frac{1}{z-1}: &&-4{\bf v}_0{\bf u}_1-2{\bf v}_1{\bf u}_1
-\frac{1}{4}{\bf u}_1\bar{\bf u}_1=0,\\
\frac{1}{z+1}: &&-4{\bf v}_0\bar{\bf u}_1+2{\bf v}_1\bar{\bf u}_1
+\frac{1}{4}{\bf u}_1\bar{\bf u}_1=0,\\
\frac{1}{(z-1)^2}: &&{\bf v}_1{\bf u}_1+\bar{\bf v}_1{\bf u}_1
+\frac{1}{4}{\bf u}_1\bar{\bf u}_1=0,\\
\frac{1}{(z+1)^2}: &&{\bf v}_1\bar{\bf u}_1+\bar{\bf v}_1\bar{\bf u}_1
+\frac{1}{4}{\bf u}_1\bar{\bf u}_1=0,\\
\frac{1}{z^2}: &&2({\bf v}_0{\bf u}_1-{\bf v}_0\bar{\bf u}_1)
+{\bf v}_1{\bf u}_1+{\bf v}_1\bar{\bf u}_1
+{\bf v}_1\bar{\bf v}_1=0,\\
\frac{1}{z}: &&3({\bf v}_0{\bf u}_1+{\bf v}_0\bar{\bf u}_1)
+2({\bf v}_1{\bf u}_1-{\bf v}_1\bar{\bf u}_1)=0.
\end{eqnarray*}
As the consequence, we get
\begin{gather}
\mathrm{Re}({\bf u}_1)~\bot~{\bf v}_0,~~
\mathrm{Im}({\bf u}_1)~\bot~\mathrm{Re}({\bf v}_1),\mathrm{Im}({\bf v}_1),\notag\\
2\mathrm{Im}({\bf u}_1){\bf v}_0+\mathrm{Re}({\bf u}_1)\mathrm{Im}({\bf v}_1)=0,\label{eq-conf3}\\
\mathrm{Re}({\bf u}_1)\mathrm{Re}({\bf v}_1)=
-\frac{1}{2}|{\bf v}_1|^2=-\frac{1}{8}|{\bf u}_1|^2.\notag
\end{gather}
Without loss of generality, we can always assume
\[
{\bf v}_0=\begin{pmatrix}0\\0\\1\\1\end{pmatrix},~
{\bf v}_1=\begin{pmatrix}1\\-\mathrm{i}\\0\\0\end{pmatrix},~
{\bf u}_1=\begin{pmatrix}a_1+b_1\mathrm{i}\\a_2+b_2\mathrm{i}\\
a_3+b_3\mathrm{i}\\a_4+b_4\mathrm{i}\end{pmatrix}.
\]
Then the conformal condition \eqref{eq-conf3} implies
\[
a_3=a_4,~b_1=b_2=0,~2(b_3-b_4)-a_2=0,~a_1=-1,~|{\bf u}_1|^2=8.
\]
Together with the isotropic condition ${\bf u}_1{\bf u}_1=0$,
we have
\[
1+4(b_3-b_4)^2=4=(b_3)^2-(b_4)^2,~a_3(b_3-b_4)=0.
\]
Thus $a_3=0,~a_2=\sqrt{3}$ and
\begin{equation}\label{eq-b3b4}
b_3=\frac{4}{\sqrt{3}}+\frac{\sqrt{3}}{4},
~b_4=\frac{4}{\sqrt{3}}-\frac{\sqrt{3}}{4}.
\end{equation}
Substitute these solutions back to \eqref{eq-decom3} and one gets
\begin{equation}\label{eq-decom4}
{\bf x}_z= \left(\frac{1}{z^3}-z\right)\begin{pmatrix}0\\0\\1\\1\end{pmatrix}
+\frac{1}{z^2}\begin{pmatrix}1\\-\mathrm{i}\\0\\0\end{pmatrix}
+ \begin{pmatrix}1\\ \mathrm{i}\\0\\0\end{pmatrix}
+\frac{1}{(z-1)^2}\begin{pmatrix}-1\\ \sqrt{3}\\b_3\mathrm{i}\\
b_4\mathrm{i}\end{pmatrix}
+\frac{1}{(z+1)^2}\begin{pmatrix}-1\\ \sqrt{3}\\-b_3\mathrm{i}\\
-b_4\mathrm{i}\end{pmatrix}.
\end{equation}
Compare this with the Weierstrass representation formula, we obtain
\begin{equation}\label{eq-dh}
\mathrm{d}h=\left(\frac{1}{z^3}-z+\frac{16\mathrm{i}}{\sqrt{3}}
\frac{z}{(z^2-1)^2}\right)\mathrm{d}z,~~
\phi\psi\mathrm{d}h=\frac{-\sqrt{3}\mathrm{i}\cdot z}{(z^2-1)^2}\mathrm{d}z,
\end{equation}
\[
\phi\mathrm{d}h=\left(\frac{1}{z^2}+\frac{\omega}{(z-1)^2}
+\frac{\omega}{(z+1)^2}\right)\mathrm{d}z,~~
\psi\mathrm{d}h=\left(1+\frac{\omega^2}{(z-1)^2}
+\frac{\omega^2}{(z+1)^2}\right)\mathrm{d}z,
\]
where $\omega=\frac{-1+\sqrt{3}\mathrm{i}}{2}$ is the third
unit root. From this we deduce the Gauss maps are
\begin{equation}\label{eq-gaussmap}
\phi=\frac{\phi\psi\mathrm{d}h}{\psi\mathrm{d}h}
=\frac{-\sqrt{3}\mathrm{i}\cdot z}{(z^2-1)^2+2\omega^2(z^2+1)},~~
\psi=\frac{\phi\psi\mathrm{d}h}{\phi\mathrm{d}h}
=\frac{-\sqrt{3}\mathrm{i}\cdot z^3}{(z^2-1)^2+2\omega^2z^2(z^2+1)}.
\end{equation}
Note that $\deg\phi=\deg\psi=4$ as we desired according to
our assumptions and the index formula.

Finally we have to verify $\phi(z)\ne\bar\psi(z)$ for any
$z\ne 0$, which is the only regularity condition need to verify.
(By the expression of ${\bf x}_z$ we know the ends are
$z=0,\infty,c,-1/c$, and there is no usual branch points.)
But in the end we will see this is not true.

To solve this problem, as in Section~3, suppose
\[
\exists \frac{1}{t}\in\mathbb{C}\backslash\{0\}, s.t.~~
\phi(z)=\bar\psi(z)=\phi(\frac{-1}{\bar{z}}).
\]
Note that $\phi(z)=\frac{1}{t}$ is indeed a fourth order algebraic equation
\begin{equation}\label{eq-}
(z^2-1)^2+2\omega^2(z^2+1)-\sqrt{3}\mathrm{i}tz
=z^4-(3+\sqrt{3}\mathrm{i})z^2-\sqrt{3}\mathrm{i}tz
-\sqrt{3}\mathrm{i}=0.
\end{equation}
The existence of non-trivial singular points is now equivalent to
the existence of two roots $z_0,-1/\bar{z}_0$ (and other two $z_1,z_2$) for the polynomial above.
By Vi\`{e}ta's formula,
\begin{eqnarray}
  z_0-\frac{1}{\bar{z}_0}+z_1+z_2 &=& 0, \label{eq-vieta5} \\
  -\frac{z_0}{\bar{z}_0}+z_1 z_2 +(z_0-\frac{1}{\bar{z}_0})(z_1+z_2) &=& -(3+\sqrt{3}\mathrm{i}),\label{eq-vieta6}\\
  -\frac{z_0}{\bar{z}_0}(z_1+z_2)+z_1 z_2 (z_0-\frac{1}{\bar{z}_0}) &=& \sqrt{3}\mathrm{i}t, \label{eq-vieta7}\\
  -\frac{z_0}{\bar{z}_0}\cdot z_1 z_2 &=& -\sqrt{3}\mathrm{i}.\label{eq-vieta8}
\end{eqnarray}
Express $z_1+z_2$ and $z_1z_2$ in terms of
$z_0,-\frac{1}{\bar{z}_0}$ using
\eqref{eq-vieta5}\eqref{eq-vieta8}.
Since $z_0=r\mathrm{e}^{\mathrm{i}\theta}$,
$-\frac{z_0}{\bar{z}_0}=-\mathrm{e}^{2\mathrm{i}\theta}$,
$z_0-\frac{1}{\bar{z}_0}=(r-\frac{1}{r})\mathrm{e}^{\mathrm{i}\theta}$.
Substitute these into \eqref{eq-vieta6}. We obtain
\[
-\mathrm{e}^{2\mathrm{i}\theta}
+\sqrt{3}\mathrm{i}\mathrm{e}^{-2\mathrm{i}\theta}
-(r-\frac{1}{r})^2\mathrm{e}^{2\mathrm{i}\theta}=-(3+\sqrt{3}\mathrm{i})
\]
Comparing the real and imaginary parts separately, one finds
\begin{gather}
r^2-1+\frac{1}{r^2}-\sqrt{3}\sin 4\theta
-3\cos 2\theta-\sqrt{3}\sin 2\theta=0,\label{eq-real}\\
\sqrt{3}\sin 2\theta=\cos 4\theta+\cos 2\theta.\label{eq-imag}
\end{gather}
Let $\lambda=\cos 2\theta$. Then \eqref{eq-imag} implies
$\lambda$ must satisfy
\begin{equation}\label{eq-imag2}
\pm\sqrt{3}\sqrt{1-\lambda^2}=2\lambda^2+\lambda-1.
\end{equation}
It is easy to see that this amounts to find the intersection
points of an ellipse and a parabola. Drawing graphs of both functions shows
that there exists a solution to \eqref{eq-imag2} such that
$1/2<\lambda_0<1$. Next we insert \eqref{eq-imag} into \eqref{eq-real} and find
\begin{equation}\label{eq-real2}
r^2+\frac{1}{r^2}=2\lambda(2\lambda^2+2\lambda+1).
\end{equation}
When $\lambda=\lambda_0\in(\frac{1}{2},1)$, since the right hand side
is greater than $2$, there is positive solution $r_0$ to the equation~\eqref{eq-real2}. This shows the only possible
example as given by \eqref{eq-decom4} (or equivalently, by
\eqref{eq-dh}\eqref{eq-gaussmap}) violates the regularity condition. This completes the proof of the non-existence result.
\end{proof}

\end{document}